\documentclass[11pt,oneside]{amsart}
\usepackage[mathscr]{eucal}
\usepackage{amssymb, amsmath, array, amscd}
\usepackage[dvips]{graphicx}
\usepackage{color}
\usepackage{epsfig}
\newtheorem{theorem}{Theorem}

\newtheorem{proposition}[theorem]{Proposition}
\newtheorem{lemma}[theorem]{Lemma}

\newtheorem{remark}[theorem]{Remark}

\begin{document}

\title[Meromorphic First Integrals]{SOME REMARKS ON MEROMORPHIC \\
FIRST INTEGRALS}

\author{Marco Brunella}

\address{Marco Brunella, Institut de Math\'ematiques de Bourgogne
-- UMR 5584 -- 9 Avenue Savary, 21078 Dijon, France}

\begin{abstract} A scholium on a paper by Cerveau and Lins Neto.
\end{abstract}

\maketitle

Our starting point is the following result, recently established by
Cerveau and Lins Neto in their paper \cite{CLN}:

\begin{theorem}\label{main}
Let ${\mathcal F}$ be a germ of holomorphic foliation on $({\mathbb
C}^2,0)$. Suppose that there exists a germ of real analytic
hypersurface $M\subset ({\mathbb C}^2,0)$ which is invariant by
${\mathcal F}$. Then ${\mathcal F}$ admits a meromorphic first
integral.
\end{theorem}

Of course, in this statement the hypersurface $M$ may be singular at
$0$, and this singularity may be even non-isolated. To say that $M$
is invariant by the foliation refers to its smooth part $M_{\rm
reg}$.

The proof given in \cite{CLN} is rather involved. There are two
cases: the {\it dicritical} case and the {\it nondicritical} one. In
the first case, the authors find a first integral by a quite
mysterious computation with power series. In the second case, they
use delicate dynamical considerations (holonomy group).

Our aim is to give an almost straightforward proof of Theorem
\ref{main}, which is based only on some general principles of
analytic geometry (in the spirit of our previous paper on a closely
related subject \cite{Bru}), together with a general (and simple)
criterion for the existence of a meromorphic first integral. This
relatively new proof will reveal the beautiful geometric structure
behind foliations tangent to real analytic hypersurfaces.

Let us also recall that Theorem \ref{main} generalizes to
codimension one foliations in higher dimensional spaces, by a
standard sectional argument \cite{M-M} \cite{CLN} (alternatively,
our arguments also generalize to higher dimensions with no
substantial new difficulty). Another possible generalization
concerns foliations defined on singular spaces, instead of ${\mathbb
C}^2$.

\section{An integrability criterion}

Let ${\mathcal F}$ be a holomorphic foliation on a domain $U\subset
{\mathbb C}^2$ containing the origin, with ${\rm Sing}({\mathcal
F})=\{ 0\}$. Set $U^\circ = U\setminus\{ 0\}$. A {\it meromorphic
first integral} is a nonconstant meromorphic function on $U$ which
is constant along the leaves of ${\mathcal F}$.

\begin{proposition}\label{criterion}
Suppose that there exists an irreducible analytic \footnote{To avoid
confusion: ``analytic'' without the ``real'' attribute means
``complex analytic''} hypersurface
$$W\subset U^\circ\times V,$$
$V$ being a neighbourhood of $0$ in ${\mathbb C}^2$, such that:
\begin{enumerate}
\item[(1)] for every $p\in U^\circ$, the fiber
$$W_p = W\cap (\{ p\}\times V)\subset V$$
is a proper analytic curve in $V$, passing
through the origin;
\item[(2)] if $p,q\in U^\circ$ belong to the same leaf of ${\mathcal
F}$, then $W_p=W_q$;
\item[(3)] the projection of $W$ to $V$ is Zariski-dense (i.e., not
contained in a curve).
\end{enumerate}
Then ${\mathcal F}$ admits a meromorphic first integral on $U$.
\end{proposition}

The germ-oriented reader should here replace $V$ with its germ at
the origin, and $W$ with its germ along $U^\circ\times\{ 0\}$.

\begin{proof}
It can be resumed as follows. We already have, by assumptions (1)
and (2), a ``first integral'', but, instead of being a meromorphic
function, it is a map which takes values into the ``space of curves
in $V$ through $0$''. Hence, roughly speaking, we shall give an
algebraic structure to such a space of curves, so that the true
meromorphic first integral will be obtained by composition of the
former ``first integral'' with a generic meromorphic function on the
space of curves. Hypothesis (3) will guarantee that such a first
integral is not identically constant. All of this is trivial if, for
instance, each $W_p$ is a line: the space of lines through the
origin is the familiar algebraic variety ${\mathbb C}P^1$. The
general case only requires some additional blow-ups.

Given a sequence of $\ell$ blow-ups
$$\pi : \widetilde V\to V$$
over the origin, denote by $D=\cup_{j=1}^\ell D_j$ the exceptional
divisor $\pi^{-1}(0)$, and set
$$\Pi = id\times\pi : U^\circ\times\widetilde V \to U^\circ\times
V.$$ Denote by $\widetilde W$ the strict transform of $W$, i.e. the
closure of the inverse image by $\Pi$ of $W\setminus (W\cap
U^\circ\times\{ 0\} )$. The trace of $\widetilde W$ on
$U^\circ\times D$ is a hypersurface (of dimension $2$), and we shall
denote by $Z$ the union of those irreducible components whose
projection to $U^\circ$ is dominant (the other components project to
curves). Thus, for $p\in U^\circ$ generic, the fiber
$$Z_p = Z\cap (\{ p\}\times D)$$
is a finite subset of $D$, which actually coincides with the trace
on $D$ of the strict transform of $W_p$ (here we have to exclude not
only those points $p$ such that $Z_p$ contains some component of
$D$, but also those points which belong to the projection of the
non-dominant components of the trace of $\widetilde W$: these are
precisely the conditions ensuring that the fiber of $\widetilde W$
over $p$ is equal to the strict transform of $W_p$).

Now, hypothesis (3) implies the following: there exists a sequence
of blow-ups $\pi :\widetilde V\to V$ over the origin such that $Z$
is {\it not} of the type $U^\circ\times\{ {\rm finite\ set} \}$.
Indeed, in the opposite case the generic curves $W_p$ would be all
unseparable by any sequence of blow-ups, i.e. they would be all
equal, and this contradicts the Zariski-density of the projection
$W\to V$ (here we use the irreducibility of $W$, and also the fact
that every $W_p$ pass through the origin).

In this way, we get an irreducible component of $D$ (say, $D_\ell$)
such that the part of $Z$ inside $U^\circ\times D_\ell$ (call it
$Z_\ell$) is dominant over $U^\circ$ and Zariski-dense over
$D_\ell$.

If $k$ is the degree of $Z_\ell\to U^\circ$, then $Z_\ell$ defines a
meromorphic map $I$ from $U^\circ$ to $D_\ell^{(k)}$, the $k$-fold
symmetric product of $D_\ell$. Such a map is not constant, but it is
constant along the leaves of ${\mathcal F}$, by hypothesis (2).
Since $D_\ell^{(k)}$ is an algebraic variety, we can find $F\in
{\mathcal M}(D_\ell^{(k)})$ such that $f=F\circ I$ is a nonconstant
meromorphic function, constant along the leaves. Finally, $f$
extends from $U^\circ$ to $U$ by Levi's theorem.
\end{proof}

\begin{remark}\label{leafspace}
{\rm Consider a foliation ${\mathcal F}$ on $U\subset{\mathbb C}^2$,
${\rm Sing}({\mathcal F})=\{ 0\}$, such that every leaf $L$ is a
so-called {\it separatrix} at $0$: $L\cup\{ 0\}$ is a proper
analytic curve in $U$. This occurs if ${\mathcal F}$ has a
meromorphic first integral having $0$ as indeterminacy point, but,
as is well known, the converse implication is far from being true,
see for instance \cite{Mou} and references therein. We have a
naturally defined subset $S$ of $U^\circ\times U$: its fiber over
$p$ is, by definition, the curve $\overline L_p = L_p\cup\{ 0\}$.
However, generally speaking this subset $S$ is {\it not} an analytic
subset, since it may be not closed.

Of course, we may take the Zariski-closure $\widehat S$ of $S$,
which however could be the full $U^\circ\times U$. If it is not the
case, i.e. if $\dim \widehat S =3$, then by Proposition
\ref{criterion} we get a meromorphic first integral, and the
converse is also true by an easy argument. Note, however, that in
this special case our Proposition \ref{criterion} is closely related
to old results by B. Kaup  and Suzuki \cite[\S 5]{Suz}, relating the
existence of first integrals with the analyticity of the graph of
the foliation.

Let us stress that, even when a first integral exists, the subset
$S$ is typically {\it not} an analytic subset, that is its
Zariski-closure $\widehat S$ may be much larger than $S$. Indeed,
the fiber of $\widehat S$ over $p$ may contain, besides $\overline
L_p$, other components $\overline L_{p_1},\ldots ,\overline
L_{p_n}$. These additional separatrices are precisely the ones which
cannot be separated from $\overline L_p$ by meromorphic first
integrals. In other words, whereas $S$ represents the (nonanalytic)
equivalence relation generated by the leaves, $\widehat S$
represents the (analytic) equivalence relation generated by level
sets of meromorphic first integrals.}
\end{remark}

There is a variant of Proposition \ref{criterion} in which the
hypothesis that every $W_p$ is a curve passing through the origin of
$V$ is replaced by a similar asymptotic hypothesis over the singular
point of the foliation. Firstly observe that if $W\subset
U^\circ\times V$ is as in Proposition \ref{criterion}, then, by
standard extension theorems, $W$ can be prolonged to an irreducible
analytic hypersurface in $U\times V$. However, it may happen that
the fiber over $0$ of this extension is not a curve, but the full
$V$; this is precisely the case in which the meromorphic first
integral has an indeterminacy point at $0$.

\begin{proposition}\label{criterionbis}
Suppose that there exists an irreducible analytic hypersurface
$W\subset U\times V$, $V$ being a neighbourhood of $0$ in ${\mathbb
C}^2$, such that:
\begin{enumerate}
\item[(1)] for every $p\in U$, the fiber
$W_p = W\cap (\{ p\}\times V)\subset V$ is a proper analytic curve
in $V$, passing through the origin when $p=0$;
\item[(2)] if $p,q\in U^\circ$ belong to the same leaf of ${\mathcal
F}$, then $W_p=W_q$;
\item[(3)] the projection of $W$ to $V$ is Zariski-dense.
\end{enumerate}
Then ${\mathcal F}$ admits a holomorphic first integral on some
(possibly smaller) neighbourhood of $0$.
\end{proposition}

\begin{proof}
It is even simpler than the previous one; in some sense, it is the
``no blow-up case''.

Take a (possibly singular) disk $D\subset V$ passing through $0$ and
intersecting $W_0$ only at $0$. Take the trace $Z$ of $W$ on
$U\times D$. Then, up to shrinking $U$, $Z$ is a hypersurface in
$U\times D$ and the projection $Z\to U$ is proper, say of degree
$k$. We thus obtain, as before, a first integral with values in
$D^{(k)}$. This last space admits a lot of holomorphic functions,
and so we get a holomorphic first integral. Thanks to hypothesis
(3), and by a suitable choice of $D$, this first integral will be
not identically constant (it is sufficient to choose $D$ highly
tangent to a branch of $W_0$).
\end{proof}

\begin{remark}\label{leafspacebis}
{\rm Consider a foliation ${\mathcal F}$ on $U\subset{\mathbb C}^2$,
${\rm Sing}({\mathcal F})=\{ 0\}$, such that there is a finite
number of separatrices and any other leaf is a proper analytic curve
in $U$. Then, on a possibly smaller $U'\subset U$, the foliation
admits a holomorphic first integral \cite{M-M}. This result can be
recasted into Proposition \ref{criterionbis}, but one needs some
further work. The idea is to look again at the subset $S\subset
U^\circ\times U$ of Remark \ref{leafspace}, and to show that the
{\it topological} closure $\overline S$ in $U\times U$ is an
analytic hypersurface, which cuts the fiber over $0$ along a curve
passing through $0$. This last curve will be the union of the
separatrices (plus the origin).

This indispensable further work can be found in \cite{Mou}. Let
$\Sigma\subset U$ be the union of the separatrices and the origin.
We may assume that the closure of each separatrix is a (singular)
disk passing through $0$ and transverse to the boundary of $U$.
According to \cite[Lemme 1]{Mou}, if $p$ is sufficiently close to
$0$, and outside $\Sigma$, then $L_p$ is a curve transverse to the
boundary of $U$. Using the finiteness of the holonomy of $L_p$
(which is an elementary fact) and Reeb stability, it is then easy to
see that the restriction of $S$ to $(U'\setminus\Sigma ')\times U$
is an analytic hypersurface, where $U'$ is a sufficiently small
neighbourhood of $0$ and $\Sigma '=\Sigma\cap U'$. Take now the
topological closure $\overline S$ of $S$ in $U'\times U$. By
standard results (Remmert-Stein), if $\overline S$ is not an
analytic hypersurface, then it must contain an irreducible component
of $\Sigma '\times U$; this is however impossible, again by
\cite[Lemme 1]{Mou} (which implies that the ${\mathcal
F}$-saturation of $U'$ cannot be the full $U$). Hence, $\overline S$
is an analytic hypersurface in $U'\times U$. By the same reason, its
fiber over $0$ cannot be the full $U$, and therefore it must
coincide with $\Sigma$. We can now apply Proposition
\ref{criterionbis}.

It is also worth observing that the fiber of $\overline S$ over a
point $p\in U'\setminus\Sigma '$ is the single leaf $L_p$. This
corresponds to the fact that the leaves outside the separatrices can
be separated by holomorphic first integrals.}
\end{remark}

\section{Complexification of real hypersurfaces}

Consider now the setting of Theorem \ref{main}: ${\mathcal F}$ is a
foliation on $U\subset{\mathbb C}^2$, singular at $0\in U$, and $M$
is a real analytic hypersurface passing through the origin and
invariant by ${\mathcal F}$. We denote by $M_{\rm reg}\subset M$ the
(open) subset of regular points, i.e. the points where $M$ is a real
analytic submanifold of dimension $3$. We assume that
$0\in\overline{M_{\rm reg}}$ (otherwise, the germ of $M$ at $0$
would not be a germ of {\it hypersurface}, as prescribed by Theorem
\ref{main}). Without loss of generality, we may also assume that $M$
is irreducible, and even that the germ of $M$ at $0$ is irreducible.

Let us recall few facts concerning complexification, see also
\cite[\S 3]{Bru}.

Denote by $U^*$ the complex manifold conjugate to $U$: it is the
same differentiable manifold, but with the opposite complex
structure; equivalently, holomorphic functions on $U^*$ are the same
as antiholomorphic functions on $U$. Remark that if $A$ is an
analytic subset of $U$, then it is analytic also as a subset of
$U^*$. As such, it will be denoted by $A^*$. In particular, every
point $p\in U$ has a ``mirror'' point $p^*\in U^*$. Similarly, if
${\mathcal F}$ is a holomorphic foliation on $U$, then it is
holomorphic also as a foliation on $U^*$, and as such it will be
denoted by ${\mathcal F}^*$. Remark that, generally speaking, the
two foliations ${\mathcal F}$ and ${\mathcal F}^*$ are {\it
different} as holomorphic foliations: the identity map $U\to U^*$
obviously conjugate ${\mathcal F}$ to ${\mathcal F}^*$, but such a
map is {\it anti}holomorphic, and not holomorphic. For example, if
$\gamma\subset L\in{\mathcal F}$ is a loop with linear holonomy
$\lambda$, then the same loop $\gamma\subset L^*\in{\mathcal F}^*$
has linear holonomy $\overline\lambda$.

In the product space $U\times U^*$ (with the product complex
structure) we have the involution
$$\jmath : U\times U^* \rightarrow U\times U^*$$
$$\jmath (p,q^*) = (q,p^*).$$
It is antiholomorphic. Its fixed point set is the diagonal $\Delta$,
and it is a totally real submanifold.

It is convenient to look at our real analytic hypersurface $M$ in
$U$ as a subset of the diagonal:
$$M\subset \Delta \subset U\times U^*.$$
Then, $M$ can be complexified: there exists a neighbourhood
$\widehat U\subset U\times U^*$ of the diagonal and an irreducible
complex analytic hypersurface $M^{\mathbb C}$ in $\widehat U$ such
that $$M^{\mathbb C}\cap\Delta = M.$$ Up to restricting $U$ around
the origin, we may assume $\widehat U = U\times U^*$. Remark that
$$\jmath (M^{\mathbb C}) = M^{\mathbb C} \qquad {\rm and}\qquad {\rm
Fix}(\jmath\vert_{M^{\mathbb C}})=M.$$

Actually, this complexification can be done on any real analytic
subset. In particular, we can start with a complex analytic subset
$A\subset U$, and look at it as a subset of $\Delta$, thus
forgetting its complex analytic structure and retaining only its
real analytic one. Its complexification is then simply the product
$A\times A^*$ (which could be pompously called ``complexification of
the decomplexification of $A$'').

Consider now the projection
$$pr : M^{\mathbb C}\rightarrow U$$
to the first factor, and for every $p\in U$ set
$$M_p^{\mathbb C} = pr^{-1}(p).$$
It is an analytic subset of $U^*$.

\begin{lemma}\label{curve}
Up to shrinking $U$ around the origin, we have: for every $p\in
U^\circ$, $M_p^{\mathbb C}$ is a (nonempty) curve in $U^*$.
\end{lemma}

\begin{proof}
The irreducibility of $M^{\mathbb C}$ implies that the set of points
of $U$ over which the fiber is two-dimensional (i.e., the full
$U^*$) is discrete. Hence, up to shrinking $U$, we get that
$M_p^{\mathbb C}$ is at most one-dimensional for every $p\in
U^\circ$ (note that a shrinking of $U$ implies a simultaneous
shrinking of $U^*$, but this is not a problem).

Obviously $M_p^{\mathbb C}$ cannot contain isolated points, because
$M^{\mathbb C}$ is a hypersurface. Therefore, it remains to show
that it is not empty. Of course $M_0^{\mathbb C}$ is not empty, for
any choice of $U$, and we can distinguish two cases:

(a) $M_0^{\mathbb C}=U^*$, i.e. $M^{\mathbb C}$ contains $\{
0\}\times U^*$. Because $M^{\mathbb C}$ is $\jmath$-invariant, this
means that also the horizontal fiber $U\times\{ 0^*\}$ is fully
contained in $M^{\mathbb C}$. As a consequence, every $M_p^{\mathbb
C}$, $p\in U^\circ$, is a curve which, moreover, passes through the
origin.

(b) $M_0^{\mathbb C}$ is a curve in $U^*$. Then, by a standard
result (Remmert's Rank Theorem), the map $pr$ is open, and hence
surjective for a suitable choice of $U$.
\end{proof}

We shall see that case (a) corresponds to the dicritical case, and
case (b) to the nondicritical one.

Recall now that we have a holomorphic foliation ${\mathcal F}$ on
$U$, leaving $M$ invariant.

\begin{lemma}\label{leaf}
For every $p\in U^\circ$, the curve $M_p^{\mathbb C}\subset U^*$ is
invariant by ${\mathcal F}^*$. Moreover, if $p$ and $q$ belong to
the same leaf, then $M_p^{\mathbb C} = M_q^{\mathbb C}$.
\end{lemma}

\begin{proof}
This is basically \cite[Lemma 3.1]{Bru}, but let us explain it in a
slightly different manner.

On $U\times U^*$ we have the foliation (of dimension $2$) ${\mathcal
F}\times{\mathcal F}^*$. It is nonsingular on $U^\circ\times
U^{\circ *}$, and its leaf through $(p,q^*)$ is
$$L_{p,q^*} = L_p\times L_{q^*}^*$$
where the first factor is the leaf of ${\mathcal F}$ through $p$ and
the second factor is the leaf of ${\mathcal F}^*$ through $q^*$. In
particular, if $(p,p^*)\in\Delta$, then $L_{p,p^*} = L_p\times
L_{p^*}^*$, and this is also the complexification of
$L_p\subset\Delta$ (here $L_p$ is not yet properly embedded, but it
doesn't matter for the next arguments). It follows that if we take a
leaf $L_p$ contained in $M$, then the leaf $L_{p,p^*}$ is contained
in $M^{\mathbb C}$.  Thus we have found a continuum of leaves of
${\mathcal F}\times{\mathcal F}^*$ which are contained in
$M^{\mathbb C}$, and therefore we get that $M^{\mathbb C}$ is {\it
invariant} by ${\mathcal F}\times{\mathcal F}^*$ (this is not a
surprise, for this last foliation can be understood as the
complexification of the decomplexification of ${\mathcal F}$, which
leaves $M$ invariant).

As a consequence of this, if $L$ is any leaf of ${\mathcal F}$, its
preimage $pr^{-1}(L)\subset M^{\mathbb C}$ is a union of leaves of
${\mathcal F}\times{\mathcal F}^*$ (plus possibly some singular
point on $U^\circ\times\{ 0^*\}$), i.e. it is of the form $L\times
(L_1^*\cup\ldots \cup L_n^*)$ for suitable leaves $L_j^*$ of
${\mathcal F}^*$ (plus possibly some singular point). But this is
precisely the assertion of the lemma.
\end{proof}

\begin{remark}\label{double}
{\rm Without assuming the existence of ${\mathcal F}$, the same
argument shows the following: if $M\subset U$ is any real analytic
Levi-flat hypersurface, then on $M^{\mathbb C}$ we have a
two-dimensional foliation whose leaves are products of horizontal
and vertical fibers of $M^{\mathbb C}$. Here the essential point is
that if we take a horizontal fiber and a vertical fiber of
$M^{\mathbb C}$, passing through the same point of $M^{\mathbb C}$,
then their product is still contained in $M^{\mathbb C}$. This is a
remarkable symmetry property of $M^{\mathbb C}$, and of course it is
a manifestation of the Levi-flatness of $M$. This foliation appears
also in \cite{CLN}, as complexification of the Levi foliation, but
the authors obtain the properness of leaves only after a long tour.}
\end{remark}

\begin{remark}\label{multiple}
{\rm The fact that $M_p^{\mathbb C}$ may contain several leaves of
${\mathcal F}^*$ should be compared with the phenomenon described in
Remark \ref{leafspace}. Note also that on a neighbourhood of $M_{\rm
reg}$ we have a Schwarz reflection at the level of the leaf space
\cite[p. 669]{Bru}. If $p$ is close to $M_{\rm reg}$, then
$M_p^{\mathbb C}$ contains the Schwarz reflection of $L_p$ (which
must be understood as a leaf of ${\mathcal F}^*$). The fact that
$M_p^{\mathbb C}$ is defined for every $p$ could be interpreted as a
sort of ``globalization'' of that Schwarz reflection, and the fact
that $M_p^{\mathbb C}$ contains several leaves suggests that
``reflection'' should be replaced by ``correspondence''.

For example, suppose that ${\mathcal F}$ is the radial foliation
($zdw-wdz=0$), so that $M$ corresponds to a real algebraic curve
$\gamma\subset{\mathbb C}P^1$ ($=$ the space of leaves of ${\mathcal
F}$). The complexification of $\gamma$ is a complex algebraic curve
$\gamma^{\mathbb C}\subset {\mathbb C}P^1\times{\mathbb C}P^{1*}$,
which gives an antiholomorphic correspondence of ${\mathbb C}P^1$
with itself. \footnote{We say that a real analytic curve
$\gamma\subset{\mathbb C}P^1$ is {\it real algebraic} if its
complexification $\gamma^{\mathbb C}$, which is in principle defined
only on a neighbourhood of the diagonal, extends to the full
${\mathbb C}P^1\times{\mathbb C}P^{1*}$. With this definition, it is
easy to see that a radial Levi-flat hypersurface $M$ is analytic at
$0$ if and only if the corresponding $\gamma$ is real algebraic; the
complex curve $\gamma^{\mathbb C}$ is then the trace of $\widetilde
M^{\mathbb C}$ on ${\mathbb C}P^1\times{\mathbb
C}P^{1*}\subset\widetilde U\times \widetilde U^*$, where $\widetilde
M$ is the strict transform of $M$ in $\widetilde U =$ the blow-up of
$U$ at $0$.} }
\end{remark}

We can now immediately complete the proof of Theorem \ref{main}.
Firstly note that, by Lemmata \ref{curve} and \ref{leaf}, every leaf
of ${\mathcal F}$ is properly embedded in $U^\circ$. If
$M_0^{\mathbb C}$ is the full $U^*$ then, as observed in the proof
of Lemma \ref{curve}, every $M_p^{\mathbb C}$, $p\not= 0$, is a
curve through the origin, and so we can apply Proposition
\ref{criterion} to get a meromorphic first integral. If
$M_0^{\mathbb C}$ is a curve, then it is a curve through the origin,
by symmetry, and so we can apply Proposition \ref{criterionbis}
(actually, in that Proposition the requirement that $W_0$ passes
through $0$ can be obviously replaced by $W_0\not= \emptyset$).

Let us conclude with a question. In the setting of Theorem
\ref{main}, consider firstly the case where we have a (primitive)
holomorphic first integral $f$, with $f(0)=0$. It is then easy to
see that $M=f^{-1}(\gamma )$, with $\gamma\subset{\mathbb C}$ a real
analytic curve passing through the origin. Indeed, we obviously have
$M=\widehat f^{-1}(\widehat\gamma )$ where $\widehat f$ is the
projection to the space of leaves $\Sigma$ (a non-Hausdorff Riemann
surface \cite{Mou} \cite{Suz}) and $\widehat\gamma\subset\Sigma$ is
a real analytic curve. Moreover, $f=e\circ\widehat f$, where
$e:\Sigma\to V\subset{\mathbb C}$ is the map which collapses
nonseparated points. However, due to the special structure of
$\Sigma$ in this case (there is only a finite set of nonseparated
points, all sent to $0$ by $e$) we certainly have $\widehat\gamma
=e^{-1}(\gamma )$ for some $\gamma\subset V$, and so
$M=f^{-1}(\gamma )$. Now, consider the dicritical case, where the
first integral $f$ is only meromorphic, and $0$ is an indeterminacy
point. Can we find a real algebraic curve $\gamma\subset{\mathbb
C}P^1$ such that $M=f^{-1}(\gamma )$? The problem here is that the
collapsing map $e:\Sigma\to{\mathbb C}P^1$ is much more complicated,
and in principle the curve $\widehat\gamma\subset\Sigma$ could be
not of the form $e^{-1}(\gamma )$, i.e. there could exist two
unseparable leaves $L,L'$ with $L$ in $M$ but not $L'$. Of course,
we can set $\gamma = e(\widehat\gamma )$ and take $f^{-1}(\gamma )$,
but this last one could be reducible, and our initial $M$ could be
only one irreducible component of it.

\begin{figure}[h]
\begin{center}
\includegraphics[width=10cm,height=4cm]{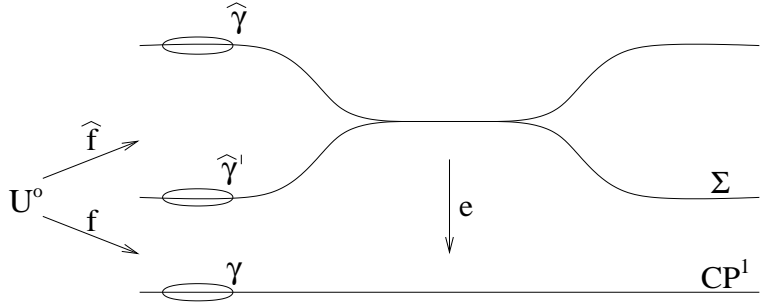}
\caption{Can $M=\widehat f^{-1}(\widehat\gamma )$ and $M'=\widehat
f^{-1}(\widehat\gamma ' )$ be irreducible components of
$f^{-1}(\gamma )$?}
\end{center}
\end{figure}

\end{document}